\documentclass{amsart}

\usepackage{amsmath}
\usepackage{amssymb}
\usepackage{amsthm}
\usepackage[english]{babel}
\usepackage[utf8]{inputenc}
\usepackage{enumitem}

\newcommand{\F}{\mathbb{F}}
\newcommand{\N}{\mathbb{N}}
\newcommand{\bA}{\mathbf{\alpha}}
\newcommand{\bB}{\mathbf{\beta}}
\newcommand{\cS}{\mathcal{S}}

\newcommand{\cP}{\mathcal{P}}
\renewcommand{\S}{\mathcal{S}}
\newtheorem{theorem}{Theorem}
\newtheorem{corollary}{Corollary}
\newtheorem{lemma}{Lemma}
\newtheorem{definition}{Definition}
\newtheorem{prop}[lemma]{Proposition}

\newtheorem{problem}{Problem}
\theoremstyle{remark}
\newtheorem*{remark}{Remark}
\newtheorem*{example}{Example}
\usepackage{color}

\usepackage{mathtools}
\usepackage{todonotes}

\title[Expansion complexity of sequences]{Algebraic dependence in generating functions and expansion complexity}
\author{Domingo G\'omez-P\'erez}
\address{D.G.-P.: Department of Mathematics, University of Cantabria, Santander 39005, Spain}
\email{domingo.gomez@unican.es}

\author{L\'aszl\'o M\'erai}
\address{L.M. Johann Radon Institute for Computational and Applied Mathematics, Austrian Academy of Sciences,  Altenberger Stra\ss e 69, A-4040 Linz, Austria} 
\email{laszlo.merai@oeaw.ac.at}

\subjclass{11T71, 11Y16,  94A60, 94A55, 68Q25}

\begin{document}

\begin{abstract}
In 2012, Diem introduced a new figure of merit for cryptographic sequences
called expansion complexity. Recently, a series of paper has been published for analysis of expansion complexity and for testing sequences in terms of this new measure of randomness.  In this paper, we continue this analysis. First we study the expansion complexity in terms of the Gr\"obner basis of the underlying polynomial ideal. Next, we prove bounds on the expansion complexity for random sequences. Finally, we study the expansion complexity of sequences defined by differential equations, including the inversive generator.
\end{abstract}

\keywords{pseudorandom sequence, expansion complexity,  Gr\"obner basis, inversive generator}

\maketitle

\section{Introduction}
For a sequence $\cS=(s_n)_{n=0}^\infty$ over the finite field $\F_q$ of
$q$ elements, we define its \emph{generating function} $G(x)$ of $\cS$
by
\begin{equation*}
  G(x) = \sum_{n=0}^{\infty}s_nx^n,
\end{equation*}
viewed as a formal power series over $\F_q$.

A sequence $\cS$ is called \emph{expansion sequence} or \emph{automatic sequence} if its
generating function satisfies an algebraic equation
\begin{equation}\label{eq:h}
  h(x,G(x))=0
\end{equation}
for some nonzero polynomial $h(x,y)\in\F_q[x,y]$. Clearly, the polynomials
$h(x,y)\in\F_q[x,y]$ satisfying \eqref{eq:h} form an ideal in
$\F_q[x,y]$. This ideal is called the \emph{defining ideal} and it is
a principal ideal generated by an irreducible polynomial, see
\cite[Proposition~4]{di12}.

Expansion sequences  can be efficiently computed from a relatively
short subsequence via the generating polynomial of its defining
ideal~\cite[Section~5]{di12}.

\begin{prop}\label{prop:expansion_seq}
Let $\cS$ be an expansion sequence and let $h(x,y)$ be the generating
polynomial of its defining ideal.
The sequence $\cS$ is uniquely determined by $h(x,y)$ and its initial
sequence of length $(\deg h)^2$.
Moreover, $h(x,y)$ can be computed in polynomial time (in $\log q
\cdot \deg h $) from an initial sequence of length $(\deg h)^2$.
\end{prop}
Based on Proposition \ref{prop:expansion_seq}, Diem~\cite{di12}
defined the $N$th expansion complexity in the following way.
For a positive integer $N$, the {\em $N$th expansion complexity}
$E_N=E_N(\cS)$ is $E_N=0$ if $s_0=\ldots=s_{N-1}=0$ and otherwise
the least total degree of a nonzero polynomial $h(x,y)\in \F_q[x,y]$ with
\begin{equation}\label{eqcon}
 h(x,G(x))\equiv 0 \mod x^N.
\end{equation}

For recent results on expansion complexity we refer to \cite{merai2018expansion,merai2017expansion}.
For example, it was pointed out in \cite{merai2018expansion}, that small expansion complexity does not imply high predictability in the sense of Proposition~\ref{prop:expansion_seq}.

\begin{example}\label{ex:1}
Let $\cS$ be a sequence  over the finite field $\F_p$ ($p\geq 3$) with
initial segment $\cS=000001\dots$ and generating function $G(x)\equiv
x^5 \mod x^6$. Then its 6th expansion complexity is $E_6(\cS)=2$
realized by the polynomial $h(x,y)=x\cdot y$. However, the first 4
elements do not determine the whole initial segment with length 6.
\end{example}

In order to achieve the predictability of sequences in terms of
Proposition~\ref{prop:expansion_seq}, one needs to require that the
polynomial $h(x,y)$ satisfying \eqref{eqcon} is
\emph{irreducible}. This observation leads to  the
\emph{i(rreducible)-expansion complexity} of a sequence. Accordingly,
for a positive integer $N$, the {\em $N$th i-expansion complexity}
$E^*_N=E^*_N(\cS)$ is $E_N^*=0$ if $s_0=\ldots=s_{N-1}=0$ and
otherwise the least total degree
of an  irreducible polynomial $h(x,y)\in \F_q[x,y]$ with
\eqref{eqcon}.

See \cite{merai2018expansion} for more details for expansion and i-expansion complexity.

\bigskip

In this paper we first give bounds on the expansion and i-expansion complexity in terms of the Gr\"obner basis of the ideal of polynomials~\eqref{eqcon} in Section~\ref{sec:groebner}. 
In Section~\ref{sec:prob} we study the typical value of expansion complexity for random sequences. Finally, in Section~\ref{sec:seq} we study the expansion complexity of sequences defined by differential equations. An example of such a sequence is the so-called explicit inversive generator.

\section{Expansion complexity and Gr\"obner bases}\label{sec:groebner}

In this section we determine the expansion and i-expansion complexity of a sequence in terms of the Gr\"obner basis of its defining ideal.

\subsection{A brief introduction to Gr\"obner bases}
In the following section, we give a brief introduction of Gr\"obner bases
with special emphasis in properties. For a more complete introduction, we
recommend to consult the introductory books of Eisenbud~\cite{Eis} and zur Gathen~\cite{von2013modern}. 
In this section we focus only on polynomials with 2 variables and recall the basic notion just for this special case.

For vectors of integer components $\bA = (\alpha_1,\alpha_2)$ define $|\bA|=\alpha_1+\alpha_2$. 
The \emph{graded lexicographical ordering}, denoted by $<_{grlex}$, is defined as $\bA
<_{grlex} \bB $ for vectors  $\bA = (\alpha_1,\alpha_2)$ and
$\bB = (\beta_1,\beta_2)$ if $|\bA|<|\bB|$ or  $|\bA|=|\bB|$ and $\alpha_2<\beta_2$.

We will use the following notation: Let  $C=\sum_{\alpha_1,\alpha_2}c_{\alpha_1,\alpha_2}x^{\alpha_1}y^{\alpha_2}$  be a nonzero
polynomial with each $c_{\alpha_1,\alpha_2}\not= 0$ and $I \subset \F_{q}[x,y]$.
Then,
\begin{enumerate}[label=(\alph*)]
\item $ LE(C)=leadexp(C)$ is the largest exponent vector $\alpha$ in
  $C$ with respect to $<_{grlex}$.

\item $LM(C)$  denotes the leading monomial of  $C$  so
  if $LE(C)=(\alpha_1,\alpha_2)$, then $LM(C)=x^{\alpha_1}y^{\alpha_2}$.

\item $LC(C)$  denotes the coefficient  of $LM(C)$.
  In other words, the so called leading term of $C$  is
  $LC(C)LM(C)$.

\item $LE(I)= \left\{ LE(C) \ | \ 0 \neq C \in I \right\}
  \subseteq \N_0^2$. (Note that if $I=\{0\}$, then $LE(I)=\emptyset $.)

\item $LM(I)= \left\{ LM(C) \ | \ 0 \neq C \in I \right\} =
  \left\{x^{\alpha_1}y^{\alpha_2} \ | \ (\alpha_1,\alpha_2) \in LE(I)\right\}$.
  (If $I=\{0\}$, then $LM(I)=\emptyset$.)

\item\label{item:linear} For $C(x,y)\in\F_q[x,y]$ with $|LE(C)| \geq 2$  and $a,b\in\F_q$ we have $LC(C(x,y))=LC(C(x,y+ax+b))$ and $LM(C(x,y))=LM(C(x,y+ax+b))$ with respect to $<_{grlex}$.
\end{enumerate}

\begin{definition}Let $\cP= \{P_1, \ldots ,P_\ell\}\subset\F_q[x,y]$ and write $I=\langle P_1, \ldots ,P_\ell \rangle$. $\cP$ is a \emph{Gr\"obner basis} for $I$ with respect to $<_{grlex}$ if $\langle LM(P_1), \ldots , LM(P_{\ell})\rangle = \langle LM\left(I\right)\rangle$. If $LC(P_i)=1$ for $ i=1, \ldots, \ell$ and $LM(P_i)$ does not divide any term of  $P_j$  for   $i\neq j$,  then   $\cP$ is a \emph{reduced Gr\"obner basis} for $I$  with respect to  $<_{grlex}$.
\end{definition}
It is known that for any ideal $I$, there exists  $\{P_1,\ldots, P_{\ell}\}$ 
that is a reduced Gr\"obner basis with respect to $<_{grlex}$ and this basis is unique,
apart from permutations of the elements.

The following corollary directly follows from Property \ref{item:linear}.

\begin{corollary}\label{cor:linear}
 Let $\cP=\{P_1(x,y),\ldots, P_{\ell}(x,y)\}$ be a reduced Gr\"obner basis for $\langle \cP \rangle$ with respect to $<_{grlex}$. If $|LE
 (P_i)| \geq 2$  for all $i=1,\dots, \ell$, then for any $a, b\in\F_q$, $\cP'=\{P_1(x,y+ax+b),\ldots, P_{\ell}(x,y+ax+b)\}$ is a reduced Gr\"obner basis for $\langle \cP' \rangle$. 
\end{corollary}

\subsection{Main results on expansion complexity and Gr\"obner bases}

For a sequence $\cS=(s_n)_{n=0}^\infty$ and  $N\geq 1$, let $G_N(x)\in\F_q[x]$ be the \emph{generating polynomial} of the truncated sequence $(s_n)_{n=0}^{N-1}$, that is,
$$
G_N(x)=\sum_{n=0}^{N-1}s_nx^n.
$$
Clearly, $G(x)\equiv G_N(x)\mod x^N$.

The polynomials $h(x,y)$ satisfying~\eqref{eqcon} form an ideal $I$ generated by $I = \langle y-G_N(x), x^{N}\rangle$. We prove the following result which  makes a link between the expansion and i-expansion complexity and the Gr\"obner basis of $I$.

\begin{theorem}\label{prop:bound}
Given any sequence $\cS$ over $\F_q$ let $\mathcal{P}=\{P_1,\ldots, P_{\ell}\}$ be a reduced Gr\"obner basis for $\langle y-G_N(x),x^{N} \rangle$ with respect to $<_{grlex}$. Then
  \begin{equation*}
    E_N(\cS)=\min \{|LE(P_1)|,\ldots, |LE(P_\ell)|\},\quad
  \end{equation*}
  and
 \begin{equation*}
    E_N^{*}(\cS)\leq \min \{|LE(P_i)|: P_i\in\mathcal{P} \ \text{is irreducible}\}.
 \end{equation*}
  \end{theorem}

As a consequence, we have the following bounds on the i-expansion complexity:
$$\min \{|LE(P_1)|,\ldots, |LE(P_\ell)|\}\leq E_N^{*}(\cS)\le \max \{|LE(P_1)|,\ldots, |LE(P_\ell)|\}.$$
\begin{remark}
    From a Gr\"obner basis with respect to a lexicographic order
    one can compute the
    Gr\"obner basis of the same ideal  with respect to the graded
    lexicographical using 
    the FGLM algorithm~\cite{faugere1993efficient}.
    The computational complexity of the algorithm, from an ideal generated by
    $I = \langle y-G_N(x), x^{N}\rangle$
    is $O\left(N^3 \right)$ field
    operations~\cite[Proposition 4.1]{faugere1993efficient}.
    Thus one can find the polynomials $P_1,\dots,
P_\ell$ in Theorem~\ref{prop:bound}, and compute the  expansion and
i-expansion complexity  in at most $ N^3(\log q)^{O(1)}$ binary
operations.
\end{remark}

\begin{proof}
In order to prove the first part, observe that for any polynomial $h(x,y)$ satisfying~\eqref{eqcon} we have $LM(P_i)\leq_{grlex} LM(h)$ for some $i$, so $\deg P_i\leq \deg h(x,y)$.

For the second part, if $s_n=0$ for $2\leq n\leq N-1$, then the result is immediate. Otherwise, we can reduce it to the case when $s_0 = s_1 = 0$. If the non-zero polynomial $h(x,y)$ satisfies~\eqref{eqcon}, then $h_1(x,y)= h(x,y+s_0+s_1x)$ is a polynomial with $\deg h = \deg h_1$ and 
$$
 h_1\left(x,\sum_{n=2}^{N-1}s_nx^n\right) = h\left(x,G_N(x)\right) \equiv 0 \mod x^N.
$$

As $E_N(\S)\geq 2$, we have $|LE(P_1)|,\dots,|LE(P_\ell)|\geq 2 $ by
the first part of the theorem. Then by Corollary~\ref{cor:linear}
  the reduced Gr\"obner basis changes according to the linear
  transform of the variables $y \rightarrow y+s_0+s_1x$. Moreover, the
  irreducibly of polynomials $h(x,y)$ and $P_1,\dots, P_\ell$ does not
  changes under this transformation. Evenmore, because the definition
  of $<_{grlex}$, applying that linear transformation to $P_1,\ldots,
  P_{\ell}$ results in a Gr\"obner basis with respect to $<_{grlex}$.

Now, we are going to show that one of the polynomials $P_1,\ldots, P_{\ell}$ must
be irreducible. Suppose contrary, that all the polynomials $P_1,\ldots, P_{\ell}$ are reducible, so for all $i=1,\ldots, \ell,$
$$
P_i(x,y) = R_i(x,y)T_i(x,y), \quad |LE(R_i)|,|LE(T_i)|\geq 1 \quad \text{for} \quad i=1,\ldots, \ell.
$$
As $P_i$ belongs to the reduced Gr\"obner basis of $\langle y-G_N(x), x^N\rangle$, we have  $T_i(x,G_N(x))\not\equiv 0\mod x^N$ and so
\begin{equation*}
 R_i(x,G_N(x))\equiv 0 \mod x.
\end{equation*}
Since $s_0=s_1=0$, the smallest degree term of $G_N(x)$ has degree at least two, so we must have
$ R_i(x,y)\in \langle x, y\rangle $. Similarly, we also get $ T_i(x,y)\in \langle x, y\rangle $.
Write
\begin{equation*}
    R_i(x,y) = y q_1(x,y) + x r_1(x),\quad
    T_i(x,y) = y q_2(x,y) + x r_2(x).
\end{equation*}
Then $R_i(x,y)T_i(x,y)\in \langle y^2, yx, x^2\rangle$, so $I =
\langle y-G_N(x),x^N\rangle = \langle R_1T_1,\ldots,
R_{\ell}T_\ell\rangle\subset \langle y^2, yx, x^2\rangle$. However,
$y-G_N(x)\not\in \langle y^2, yx, x^2\rangle $, a contradiction.  
\end{proof}

\section{A probabilistic result}\label{sec:prob}
In this section we study the $N$th expansion complexity for random sequences. We prove, that for such sequences the $N$th expansion complexity is large.

Let $\mu_q$ be the uniform probability measure on $\F_q$ which assigns the measure $1/q$ to each element of $\F_q$. Let $\F_q^{\infty}$ be the sequence space over $\F_q$
and let $\mu_q^{\infty}$ be the complete product probability measure
on $\F_q^{\infty}$ induced by $\mu_q$. We say that a property of
sequences $\mathcal{S} \in \F_q^{\infty}$
holds $\mu_q^{\infty}$-\emph{almost everywhere} if it holds for a set of sequences $\mathcal{S}$ of $\mu_q^{\infty}$-measure $1$. We may view such a property as a typical
property of a random sequence over $\F_q$.

\begin{theorem}\label{thm:probabilistic}
 We have
 \[
  \liminf_{N \to \infty} \, \frac{E_N({\cS})}{N^{1/2}} \geq \frac{\sqrt{2}}{2}
\qquad \mu_q^{\infty}\mbox{-almost everywhere}.
 \]
\end{theorem}
We remark, that Theorem~\ref{thm:probabilistic} is the corrected form of \cite[Theorem~4]{merai2017expansion}. 
In \cite{merai2017expansion}, the authors used \cite[Proposition~7]{di12},  which requires the irreducibly property, and consequently, it holds for the i-expansion complexity instead for the expansion complexity, see \cite[Theorem~2]{merai2018expansion}. Theorem~\ref{thm:probabilistic} gives now a lower bound on the expansion complexity of typical sequences.

\begin{proof}
First we fix an $\varepsilon$ with $0 < \varepsilon < 1$ and we put
\begin{equation*}
b_N= \lfloor (1-\varepsilon)(N/2)^{1/2} \rfloor \qquad \mbox{for } N=1,2,\ldots .
\end{equation*}
Then
\begin{equation}\label{eq:b_N}
b_N\geq 1 \quad \text{and} \quad \binom{b_N +2}{2} \le  (1-\varepsilon_0)N
\end{equation}
for some positive $\varepsilon_0$ if $N$ is large enough.
For such $N$ put
\[
A_N=\{{\cS} \in \F_q^{\infty}: E_N({\cS}) \le b_N\}.
\]
Since $E_N({\cS})$ depends only on the first $N$ terms of ${\cS}$, the measure $\mu_q^{\infty}(A_N)$ is given by
\begin{equation} \label{equq}
\mu_q^{\infty}(A_N) =q^{-N} \cdot \# \{{\cS} \in \F_q^N: E_N({\cS}) \le b_N\}.
\end{equation}

If $\S\in\F_q^N$ is a sequence with $E_N(\S)\leq b_N$, there is a
polynomial $h(x,y)$ with degree at most $b_N$ with \eqref{eqcon}. Write
$h(x,y)=h_1(x,y)\cdots h_k(x,y)$ with $h_i(x,y)$ irreducible factor, then 
\begin{equation}\label{eq:h_i}
 h_i(x,G(x))\equiv 0 \mod x^{N_i}, \ (1\leq i\leq k)\quad \text{with} \quad  N_1+\dots+N_k=N.
\end{equation}
Now
\begin{align*}
  \frac{1}{k}\sum_{j=1}^{k}\left (N_j - \binom{\deg h_j+2}{2}\right ) &\ge \frac{N -
    (\sum_{j=1}^k  \binom{\deg h_j+2}{2})}{k} \ge \frac{N-\binom{b_N+2}{2}}{k}\\
  &\geq \frac{\varepsilon_0 N}{b_N}\ge \varepsilon_0\sqrt{N}
\end{align*}
by the choice of $b_N$. So $N_j - \binom{\deg h_j+2}{2} \ge \varepsilon_0\sqrt{N}$ for some $1\leq j\leq k$.
Without loss of generality, we can suppose that $j=1$.

We estimate the cardinality of $A_N$ by the number of such sequences
that
$$
h_1(x,G(x))\equiv 0 \mod x^{N_1}.
$$
Write $\S=(\S_1,\S_2)\in\F_q^N$ with $\S_1\in\F_q^{N_1}$ and
$\S_2\in\F_q^{N- N_1}$. For a fixed irreducible polynomial of degree
$d$ there are at most $d$ choices for $\S_1$ (see \cite[p. 332]{di12})
and $q^{N-N_1}$ choices for $\S_2$.
If two irreducible polynomials are
constant multiples of each other, they define the same sequences $\S_1$.

Let a polynomial $f(x,y)$ of degree $d$ be called \emph{normalized}
if in the  coefficient vector $(a_0,a_1,\dots, a_d)$ of the
homogeneous part with degree $d$ of $f$, i.e.,
\[
 a_0x^d+a_1x^{d-1}y+\dots+a_d y^d,
\]
the first nonzero element is 1.

Let $I_2(d)$ be the number of  \emph{normalized} irreducible
polynomials (with two variables) in $\F_q[x,y]$ of total degree $d$. Then by
\cite{carlitz63} we have
\[
 I_2(d)=\frac{1}{q-1}q^{\binom{d+2}{2}}+O\left(q^{\binom{d+1}{2}}\right).
\]

Thus
 \begin{align*}
   \# \{{\cS} \in \F_q^N: E_N({\cS}) \le b_N\}
   & \leq \sum_{d_1\leq b_N} \sum_{\varepsilon_0\sqrt{N} + \binom{d_1 + 2}{2} \le
     N_1\le N}d_1 I(d_1)q^{N-N_1}\\
   &\ll \sum_{d_1\leq b_N} \sum_{\varepsilon_0\sqrt{N} + \binom{d_1 + 2}{2} \le
     N_1\le N}b_N q^{\binom{d_1+2}{2}-1+ N-N_1} \\
   &\ll \sum_{d_1\leq b_N}b_N N q^{N - \varepsilon_0\sqrt{N}}\ll b_N^2 N q^{N - \varepsilon_0\sqrt{N}}.
\end{align*}
By the choice of $b_N$, we have that $\mu_q^\infty(A_N)$ is at most
$q^{-\delta  N^{1/2}}$ for some positive~$\delta$. If $N$ is large enough, then $q^{-\delta  N^{1/2}}<N^{-2}$ so
$$
\sum_{N} \mu_q^{\infty}(A_N)\leq\sum_{N}q^{-\delta N^{1/2}}\ll\sum_{N}N^{-2} < \infty.
$$
Then the Borel-Cantelli lemma  shows that the set of all ${\cS} \in
\F_q^{\infty}$ for which ${\cS} \in A_N$ for infinitely many $N$
has $\mu_q^{\infty}$-measure $0$. In other words,
$\mu_q^{\infty}$-almost everywhere we have
${\cS} \in A_N$ for at most finitely many $N$. It follows then from the definition of $A_N$ that $\mu_q^{\infty}$-almost everywhere we have
$$
E_N({\cS}) > b_N > (1-\varepsilon) (N/2)^{1/2}
$$
for all sufficiently large $N$. Therefore $\mu_q^{\infty}$-almost everywhere,
$$
\liminf_{N \to \infty} \, \frac{E_N({\cS})}{(N/2)^{1/2}} \geq (1-\varepsilon).
$$
By applying this for $\varepsilon =1/r$ with $r=1,2,\ldots$ and noting that the intersection of countably many sets of $\mu_q^{\infty}$-measure $1$ has again $\mu_q^{\infty}$-measure $1$,
we obtain the result of the theorem.
\end{proof}

\section{Sequences defined by differential equations}\label{sec:seq}

In this section we study the expansion complexity of sequences characterized by the property that their generating function satisfies certain differential equations. For $r\geq 0$ let $D^{(r)}$ denote the $r$-th Hasse derivative defined by
$$
D^{(r)}x^n=\binom{n}{r}x^{n-r}.
$$
The first Hasse derivative $D^{(1)}$ is identical to the standard derivative. Moreover, it satisfies the chain rule
\begin{equation}\label{eq:hasse}
D^{(r)}(fg)=\sum _{i=0}^{r}D^{(i)}(f)D^{(r-i)}(g)
\end{equation}
for all $f,g\in \F_q[x]$. For more details see \cite{Hasse}.

In this section we consider sequences $\cS=(s_n)$ whose generating function $G(x)$ satisfies
\begin{equation}\label{eq:diff_equation_gen}
 f_{k+1}(x) D^{(k)}\left(G(x)\right)+\dots+f_2(x)D^{(1)}\left(G(x)\right)+f_1(x)G(x)+f_0(x)=0
\end{equation}
with polynomials $f_{k+1}(x),\dots, f_{0}(x)\in\F_q[x]$.

In Theorem~\ref{thm:seq} below, we give bounds on the $N$th expansion complexity of sequences over prime fields whose generating function satisfies a first order differential equation \eqref{eq:diff_equation_gen} with small degree coefficient polynomials.

One of the most important examples for such  sequence is the \emph{explicit inversive generator} over a prime field $\F_p$, with some prime $p\geq 3$, defined by
\begin{equation}\label{eq:explicit_inverse}
s_n =
\left\{ 
\begin{array}{cl}
 (an-b)^{-1} & \text{if } an-b\not \equiv 0 \mod p\\
 0 & \text{otherwise,}
\end{array}
\right.
\end{equation}
with some $a,b\in \F_p$, $a\neq 0$. Its generating function $G_{a,b}(x)$ satisfies
$$
ax(1-x)^pG'_{a,b}(x)-b(1-x)^pG_{a,b}(x)-(1-x)^{p-1}+x^{b/a \bmod p}=0,
$$
see Corollary~\ref{cor:inversive} below.

\begin{theorem}\label{thm:seq}
Let $\cS=(s_n)$ be a sequence over $\F_p$. Assume, that its  generating function $G(x)$ satisfies
\begin{equation}\label{eq:diff_equation}
f_2(x)G'(x)+f_1(x)G(x)+f_0(x)\equiv 0 \mod x^M
\end{equation}
with $M\geq 1$ for some polynomials $f_0(x),f_1(x),f_2(x)\in\F_p[x]$ 
such that there is an $\alpha\in\overline{\F}_q$ with $f_2(\alpha)=0$, $f_1(\alpha)=0$ and $f_2'(\alpha) f_0(\alpha)\neq 0$.

Let $F=\max\{\deg f_2-1,\deg f_1, \deg f_0-1\}$.  Then
$$
E_N(\cS)(E_N(\cS)+F) \geq N \quad \text{or} \quad
E_N(\cS)\geq p \quad \text{for } \deg f_0 +1< N\leq M.
$$
\end{theorem}

Previously, only a few examples for sequences were known with large expansion complexity, all of them share the property \eqref{eq:diff_equation_gen}. Namely, the sequences of binomial coefficients $\mathcal{A}=(a_n)_{n=0}^{\infty}$, defined by
$$
a_n=\binom{n+k}{k} \mod p, \quad n=0,1,\dots
$$
for some $k\geq 0$,  whose generating function is $G_k(x)=(1-x)^{-1-k}$ by \cite[Lemma~2]{merai2017expansion}, which satisfies
$$
(x-1)G'_k(x) -(k+1)G_k(x) =0,
$$
and the explicit inversive generator defined by \eqref{eq:explicit_inverse} with $b=0$, see \cite{merai2018expansion}.

We also remark, that \eqref{eq:diff_equation} defines a linear
recurrence relation to the counter-dependent sequence  $(n\, s_n)$ in
terms of $(s_n)$ and $(n\, s_n)$.
This type of relations appears in the so called
\emph{counter-dependent nonlinear recursive pseudorandom number
  generators}.  
A counter-dependent nonlinear recursive pseudorandom number generator
is of the form: 
$$
s_n = f(s_{n-1},..,s_{n-m}, n).
$$
This class of generators was introduced by Shamir and Tsaban in order
to avoid unexpected short cycles (see Definition 2.4 of
\cite{ShamirTsaban}) for $m=1$.
Special cases of this type of generators have been studied in relation
with exponential sums and multiplicative character sums
\cite{winterhof2009,EW2006,gomez,SW2006}.
For example, sequences whose 
generating function $G(x)$ satisfies
$$
x^2(1-x) G'(x)- (1-x)^2 G(x) -(s_0-1)x+s_0 = 0
$$
coincides with the special class of sequences
proposed by Shparlinski and Winterhof~\cite{SW2006}, 
defined as $s_n = n s_{n-1} +1$.

\smallskip

In order to prove Theorem~\ref{thm:seq}, we need the following result, see \cite[Lemma~6]{di12}.
\begin{lemma}\label{lemma:Diem}
 Let $h(x,y)\in\F_q[x,y]$ be an irreducible polynomial of degree $d$ and let $\cS$ be an expansion sequence defined by $h(x,y)$. Let $f(x,y)\in\F_q[x,y]$ be a nonzero polynomial with 
 $$
 f(x,G(x))\equiv 0 \mod x^{d\cdot \deg f}.
 $$
 Then $f(x,y)$ is a multiple of $h(x,y)$.
\end{lemma}

\begin{proof}[Proof of Theorem~\ref{thm:seq}]
Put $K=\deg f_0(x)$. There is a nonzero element among $s_0,\dots, \allowbreak s_{K+1}$ and thus $E_{K+1}(\cS)\geq 1$. Indeed, if $G(x)\equiv 0 \mod x^{K+2}$, then $f_0(x)=0$ by \eqref{eq:diff_equation}, a contradiction. 

If $s_0=0$, consider the sequence $\bar{\cS}=(\bar{s}_n)$ with $\bar{s}_0=1$ and $\bar{s}_n=s_n$ for $n\geq 1$. Let $\bar{G}(x)=G(x)+1$ be the generating function of $\bar{\cS}$. Then $h(x,\bar{G}(x))\equiv 0 \mod x^N$ if and only if $h(x,G(x)+1)\equiv 0 \mod x^N$. Thus $E_N(\cS)=E_N(\bar{\cS})$ whenever $E_N(\cS)>0$. As it holds for $N\geq K+1$, we can assume that $s_0\neq 0$ and $E_1(\cS)=1$.

Now suppose that the result does not hold for some $N\geq K+2$, and fix $N$ as a minimal value such
\begin{equation}\label{eq:contradiction}
d(d+F)<N.
\end{equation}
where $d=E_N(\cS)$. We can assume, that $d<p$.
Let $h(x,y)\in\F_q[x,y]$ such that $\deg h(x,y)=d$ and $h(x,G(x))\equiv 0 \mod x^N$. First we prove, that $h(x,y)$ is irreducible. Suppose, that $h(x,y)=h_1(x,y)h_2(x,y)$ and
$$
h_1\left(x,G(x)\right)\equiv 0 \mod x^N_1, \quad h_2\left(x,G(x)\right)\equiv 0 \mod x^N_2, \quad N_1+N_2\geq N.
$$
Then by the minimality of $N$ we have
$$
  \deg h_1(\deg h_1+F)\geq N_1 \text{ and } \deg h_2(\deg h_2+F)\geq N_2.
$$
Thus 
\begin{equation}
 N_1+N_2\leq   \deg h_1(\deg h_1+F)+\deg h_2(\deg h_2+F)
\leq  d(d+F)<N,
\end{equation}
a contradiction.

Taking the derivative of the equation $h(x,G(x))\equiv 0 \mod x^N$ we get
$$
\frac{\partial h}{\partial x}(x,G(x)) + \frac{\partial h}{\partial y}(x,G(x)) G'(x) \equiv 0 \mod x^{N-1},
$$
thus multiplying it with $f_2(x)$ the we get by \eqref{eq:diff_equation} that
\begin{equation}\label{eq:diff_equation_2}
f_2(x)\frac{\partial h}{\partial x}(x,G(x))  - f_1(x)G(x)\frac{\partial h}{\partial y}(x,G(x))-f_0(x) \frac{\partial h}{\partial y}(x,G(x)) \equiv 0 \mod x^{N-1}.
\end{equation}
The degree of 
\begin{equation}\label{eq:g-def}
g(x,y)=f_2(x)\frac{\partial h}{\partial x}(x,y)  - f_1(x)y\frac{\partial h}{\partial y}(x,y)-f_0(x) \frac{\partial h}{\partial y}(x,y) \in\F_p[x,y]
 \end{equation}
is $\deg g(x,y)\leq d+F$.

Let $\bar{\cS}=(\bar{s}_n)$ be an expansion sequence defined $h(x,y)$ with $\bar{s}_n=s_n$ for $0\leq n < N$. As $d^2<N$, $\bar{S}$ is unique. Then by \eqref{eq:contradiction}, \eqref{eq:diff_equation_2} and by Lemma~\ref{lemma:Diem} we get that $g(x,y)$ is a multiple of $h(x,y)$, 
\begin{equation}\label{eq:g}
 g(x,y)=c(x,y)h(x,y)
\end{equation}
for some nonzero $ c(x,y)\in \F_q[x,y]$. Comparing the degrees of $g(x,y)$ and $c(x,y)h(x,y)$ with respect to $y$, we get $c(x,y)=c(x)\in\F_q[x]$. 

We show, that $c(\alpha)\neq 0$.
Write
$$
h(x,y)=\sum_{i=0}^k r_i(x)y^i, \quad r_i(x)\in\F_p[x], \quad 0\leq i\leq k.
$$
We can assume, that $k< p$ and $r_k(x)\neq 0$.
The coefficient of $y^k$ in $c(x)h(x,y)$ is 
\begin{equation}\label{eq:y-main}
 f_2(x)r'_k(x)-kf_1(x)r_k(x)=c(x)r_k(x).
\end{equation}

If $\alpha$ is a zero of $c$, then it's a zero of $g$ by \eqref{eq:g} and thus it's a zero of $\frac{\partial h}{\partial y}$ by \eqref{eq:g-def}. As $k<p$, 
$\alpha$ is also  a zero of $r_k$.  Let $t\geq 1$ be the multiplicity of $\alpha$ in $r_k$. As $\alpha$ is a single zero of $f_2$,  its multiplicity of the left hand side of \eqref{eq:y-main} is $t$, while its  multiplicity of the right hand side is at least $t+1$, a contradiction.

Substituting $x=\alpha$ in~\eqref{eq:g}, we get
$$
c(\alpha)h(\alpha,y)=f_0(\alpha)\frac{\partial h}{\partial y}(\alpha,y)
$$
Since $c(\alpha)\neq 0$, $h(\alpha,y)$ must be zero, otherwise it cannot be a constant multiple of its derivative. Thus the minimal polynomial of $\alpha$ divides $h(x,y)$, a contradiction.
\end{proof}

Theorem~\ref{thm:seq} allows  us to control the expansion complexity of the explicit inversive generator defined by \eqref{eq:explicit_inverse}. We remark, that for $b=0$ it was shown by G\'omez-P\'erez, M\'erai and Niederreiter that the sequence has optimal expansion complexity, see \cite{merai2018expansion}. Now we deal with the general case.

\begin{corollary}\label{cor:inversive}
Let $\cS=(s_n)$ be the explicit inversive generator  defined by \eqref{eq:explicit_inverse} with $a,b\in\F_p$, $a\neq 0$. Then we have
$$
E_N(\cS)\geq cN^{1/4} \quad \text{for } 2\leq N<p
$$
for some absolute constant $c>0$.
\end{corollary}
\begin{proof}
For $b=0$ a stronger bound follows from \cite[Theorem~8]{merai2018expansion}, thus we can assume, that $b\neq 0$.

As $G_{a,b}(x)=a^{-1}G_{1,b/a}(x)$, we can assume, that $a=1$. Write $G(x)=G_{1,b}(x)$. 
Then
$$
G(x)=\sum_{\substack{n=0 \\ n\not\equiv b \bmod p}}^\infty \frac{1}{n-b}x^{n}=x^b\sum_{\substack{n=0 \\ n\not\equiv b \bmod p}}^\infty\frac{1}{n-b}x^{n-b}.
$$
Now 
\begin{equation}\label{eq:diff_1}
\left(x^{-b} G(x) \right)'=-bx^{-b-1}G(x)+x^{-b}G'(x).
\end{equation}
On the other hand
\begin{align}\label{eq:diff_2}
\left(x^{-b} G(x) \right)'
&=\left(\sum_{\substack{n=0 \\ n\not\equiv b \bmod p}}^\infty\frac{1}{n-b}x^{n-b} \right)'=\sum_{\substack{n=0 \\ n\not\equiv b \bmod p}}^\infty x^{n-b-1}=\frac{1}{x^{b+1}}\sum_{\substack{n=0 \\ n\not\equiv b \bmod p}}^\infty x^{n} \\
&=\frac{1}{x^{b+1}}\left(\sum_{\substack{n=0}}^\infty x^{n}-\sum_{\substack{n=0}}^\infty x^{pn+b} \right)=\frac{1}{x^{b+1}}\left(\frac{1}{1-x}-x^b \frac{1}{1-x^p} \right).\notag
\end{align}
Then by \eqref{eq:diff_1} and \eqref{eq:diff_2} we get
\begin{equation}\label{eq:diff}
 x(1-x)^{p+1}G'(x)-b(1-x)^{p+1}G(x)-(1-x)^{p}+x^{b}(1-x)=0.
\end{equation}
For $N\leq b$  we have
$$
 x(1-x)G'(x)-b(1-x)G(x)-1\equiv 0 \mod x^b
$$
thus by Theorem~\ref{thm:seq} we have
\begin{equation}\label{eq:inv-1}
 E_N(\cS)(E_N(\cS)+1)\geq N \quad \text{for} \quad 2\leq N\leq b. 
\end{equation}
For $N> b$  
\eqref{eq:diff} leads to
$$
 x(1-x)G'(x)-b(1-x)G(x)-1+x^{b}(1-x)\equiv 0 \mod x^p.
$$
and by Theorem~\ref{thm:seq} we get
\begin{equation}\label{eq:inv-2}
 E_N(\cS)(E_N(\cS)+b)\geq N \quad \text{for} \quad b+3\leq N\leq p-1. 
\end{equation}
If $N\ll b$, $E_N(\cS)\gg\sqrt{N}$ by ~\eqref{eq:inv-1} and  if $N\gg b^2$, we get $E_N(\cS)\gg \sqrt{N}$ by \eqref{eq:inv-2}.  
Finally, using $E_{N+1}(\cS)\geq E_{N}(\cS)$, we get $E_N(\cS)\gg \sqrt{b}$ for $b\ll N\ll b^2$ which gives the result.
\end{proof}

\begin{remark}
The proof gives the stronger bounds on expansion complexity of the explicit inversive generator $\cS_{a,b}$ with parameters $a\in\F_p^*$, $b\in \F_p$
$$
E_{N}(\cS_{a,b})\gg \sqrt{N} \quad \text{for} \quad  N\ll b \text{ or } N\gg b^2.
$$
If the parameters $(a,b)$ are chosen uniformly from $\F_p^*\times \F_p$, then it provides a square-root bound for almost all parameters $(a,b)$ which is optimal, see \cite[Theorem~1]{merai2018expansion}.
\end{remark}

In Theorem~\ref{thm:seq} we gave lower bounds on the $N$th expansion complexity of sequences whose generating function satisfies a first order differential equation \eqref{eq:diff_equation_gen}. However, we conjecture that sequences with higher order differential equation \eqref{eq:diff_equation_gen} have also large expansion complexity.

\begin{problem}\label{prob:1}
 Let $\cS=(s_n)$ be a sequence in $\F_q$ such that its generating function $G(x)$ satisfies \eqref{eq:diff_equation_gen}. Estimate the $N$th expansion complexity $E_N(\cS)$ of the sequence $\cS$ in terms of the coefficient polynomials of \eqref{eq:diff_equation_gen}.
\end{problem}

In \cite{merai2017expansion}, M\'erai, Niederreiter and Winterhof studied the connection between the expansion and linear complexity of sequences. We recall, that the $N$th linear complexity $L_N(\cS)$ of a sequence $\cS$ over a finite field $\F_q$ is zero if $s_0=\dots=s_{N-1}=0$, otherwise the least positive $L$ such that there exist $c_0,\dots, c_{L-1}\in\F_q$ such that
\begin{equation}\label{eq:linComp}
s_{n+L}=c_{L-1}s_{n+L-1}+\dots+c_0s_n, \quad 0\leq n\leq N-L-1.
\end{equation}
They proved, that large expansion complexity implies large linear complexity 
$$
L_{N}(\cS)\geq \min\left\{ E_N(\cS)-1, \frac{N+3}{2}\right\}.
$$
They also provided  a lower bound on the expansion complexity in terms of the linear complexity, however the bound also depends on the linear recurrence relation \eqref{eq:linComp}.

Here we give lower bounds on the $N$th linear complexity of sequences with \eqref{eq:diff_equation_gen} over arbitrary (i.e. not prime) finite fields. This result along with \cite{merai2017expansion} motivates Problem~\ref{prob:1}.

\begin{theorem}\label{thm:linComp}
For polynomials $f_{k+1}(x),\dots,f_{0}(x)\in \F_q[x]$ consider the differential operator $T:\F_p[[x]]\rightarrow \F_q[[x]]$,
$$
T: G(x)\mapsto   f_{k+1}(x) D^{(k)}\left(G(x)\right)+\dots+f_2(x)D^{(1)}\left(G(x)\right)+f_1(x)G(x)+f_0(x)
$$
with coprime coefficients such that it has no rational zero.
If $\cS=(s_n)_{n=0}^{\infty}$ is a sequence over $\F_q$ such that its generating function $G(x)$ satisfies
$$
T(G(x))\equiv 0 \mod x^M,
$$
then
$$
L_N(\cS)\geq \frac{N-F+2}{k+4} \quad \text{for } N\leq M.
$$
with  $F=\max\left\{\deg f_{k+1}(x),\dots, \deg f_0(x)\right\}$.
\end{theorem}

\begin{remark}
Theorem~\ref{thm:linComp} also holds with the standard derivative instead of the Hasse derivative. Thus one can also consider the analogue of Problem~\ref{prob:1}.  
\end{remark}

\begin{proof}
For $N\leq M$ put $L=L_N(\cS)$. Then there exist polynomials $g(x),h(x)\in\F_q[x]$, $\deg g(x)<L$, $\deg h(x)\leq L$, $h(x)\neq 0$ such that
\begin{equation}\label{eq:derivative}
h(x)G(x)\equiv g(x) \mod x^N.
\end{equation}
One can choose
$$
h(x)=\sum_{i=0}^{L-1}c_ix^{L-i}\quad \text{and} \quad g(x)=\sum_{m=0}^{L-1}\left(\sum_{\ell=L-m}^L c_\ell s_{m+\ell-L} \right)x^m,
$$
where $c_{L}=-1$ and $c_0,\dots, c_{L-1}$ are the coefficients of the linear recurrence relation  \eqref{eq:linComp}.

By the chain rule \eqref{eq:hasse}, and by \eqref{eq:derivative} we get
\begin{equation}\label{eq:derivative-k}
 h^{\ell+1}(x)D^{(\ell)}\left(G(x)\right)\equiv g_\ell(x) \mod x^{N-\ell}, \quad \deg g_\ell(x)\leq (\ell+1)(L-1), \quad 0\leq \ell\leq N. 
\end{equation}

Then multiplying $T(G(x))$ by $h^{k+1}(x)$ we get
\begin{align*}
0&\equiv T(G(x))h^{k+1}\\ 
&\equiv  f_{k+1}(x)h^{k+1}(x)D^{(k)}\left(G(x)\right)+\dots+f_1(x)h^{k+1}(x)G(x)+f_0(x)h^{k+1}(x) \\
&\equiv  f_{k+1}(x)g_k(x)+\dots+f_1(x)h^{k}(x)g_1(x)+f_0(x)h^{k+1}(x) \mod x^{N-L}
\end{align*}
Whence
$$
f_{k+1}(x)g_k(x)+\dots+f_1(x)h^{k}(x)g_1(x)+f_0(x)h^{k+1}(x)=J(x)x^{N-L}.
$$
If $J(x)=0$, then $\overline{G}(x)=g(x)/h(x)$ is a zero of $T$, as \eqref{eq:derivative-k} holds for $\overline{G}(x)$ with equality, a contradiction. 

Comparing the degrees of both sides we get
$$
\max_{0\leq \ell \leq k}\left\{\deg f_{\ell+1}(x)+ \deg g_\ell(x)+\left(k+1-\ell\right)\deg h(x)  \right\}\geq N-L
$$
which gives the result.
\end{proof}

\section*{Acknowledgement}
D. G-P. is partially supported by project MTM2014-55421-P from the Ministerio de Economia y Competitividad and L.~M. is partially supported by the Austrian Science Fund FWF Project I1751-N26 and P 31762.

%  \bibliographystyle{plain}
%  \bibliography{Bibl}

\end{document}